\documentclass[amssymb,amsfonts,refcheck,12pt,verbatim,righttag]{amsart}
\usepackage{color}
\usepackage{amssymb}
\usepackage{graphics}
\usepackage{tikz}
\usepackage{multicol}
\usepackage{subfigure}

\def\R {\Bbb R}
 \numberwithin{equation}{section} 
\newcommand{\beq}{\begin{equation}}
\newcommand{\eeq}{\end{equation}}

\newcommand{\ben}{\begin{eqnarray}}
\newcommand{\een}{\end{eqnarray}}

\newcommand{\bet}{\begin{eqnarray*}}
\newcommand{\eet}{\end{eqnarray*}}

\newtheorem{thm}{Theorem}[section]
\newtheorem{lem}[thm]{Lemma}

\newtheorem{prop}[thm]{Proposition}

\newtheorem{ques}[thm]{Question}

\usepackage[colorlinks,citecolor=red,pagebackref,hypertexnames=false]{hyperref}

\def\Z{{\Bbb Z}}

\def\R {\Bbb R}
\def\N {\Bbb N}

\theoremstyle{plain}

\begin{document}
\baselineskip 16pt

\title
{A note on  Hausdorff measures of self-similar sets in $\R^d$}

\author{Cai-yun Ma}

\author{Yu-feng Wu}



\begin{abstract}
We prove that for all $s\in(0,d)$ and  $c\in (0,1)$ there exists a self-similar set $E\subset \R^d$   with   Hausdorff dimension $s$ such that $\mathcal{H}^s(E)=c|E|^s$. This answers a question raised by  Zhiying Wen \cite{Wen}.
\end{abstract}

\maketitle

\setcounter{section}{0}
\section{Introduction}
In this paper, we investigate Hausdorff measures of self-similar sets in $\R^d$. Recall that  given a finite family of contracting similitudes $\Phi=\{\phi_i\}_{i=1}^{\ell}$ on $\R^d$,
 there is a unique non-empty compact set $K\subset \R^d$ satisfying  $K=\bigcup_{i=1}^{\ell}\phi_i(K)$. We call $\Phi$  an {\it iterated function system} (IFS) of similitudes and  $K$  the {\em self-similar set} generated by $\Phi$. Moreover, given a probability vector $\mathbf{p}=(p_1,\ldots,p_{\ell})$, i.e. all $p_i>0$ and $\sum_{i=1}^{\ell}p_i=1$, there is a unique Borel probability measure $\nu$ supported on $K$ satisfying
 \[\nu=\sum_{i=1}^{\ell}p_i\nu\circ \phi_i^{-1}.\] We call $\nu$  the {\em self-similar measure} generated by $\Phi$ and $\mathbf{p}$. We refer the reader to \cite{hutchinson,falconer} for the examples and detailed properties of self-similar sets and self-similar measures.

We say that $\Phi$ satisfies the {\em strong separation condition} (SSC) if $\phi_i(K)\cap \phi_j(K)=\emptyset$ for all distinct $i,j\in \{1,\ldots,\ell\}$. Similarly, we say that  $\Phi$ satisfies the {\em open set condition} (OSC) if there exists a non-empty bounded open set $V\subset \R^d$ such that $\phi_1(V),\ldots, \phi_{\ell}(V)$ are disjoint subsets of  $V$. Under   the OSC,  it is well-known that the Hausdorff dimension of $K$, denoted by   $\dim_{\rm H}K$, equals the similarity dimension of $\Phi$, i.e. the positive number $s$ satisfying $\sum_{i=1}^{\ell}r_i^s=1$, where $r_i\in(0,1)$ is the contraction ratio of $\phi_i$, $i=1,\ldots,\ell$. Also, the $s$-dimensional Hausdorff measure of $K$ satisfies that   $0<\mathcal{H}^s(K)\leq|K|^s$, here and afterwards for $A\subset \R^d$, $|A|$ stands for the diameter of $A$.   Moreover, it is known that
\begin{equation}\label{eq:mu and Hs}
\mathcal{H}^s|_K=\mathcal{H}^s(K)\mu,
\end{equation}
where 
$\mathcal{H}^s|_K$  stands for the restriction of the $s$-dimensional Hausdorff measure on  $K$,  $\mu$ is the self-similar measure generated by $\Phi$ and the probability vector $(r_1^s,\ldots, r_{\ell}^s)$.   We call $\mu$ the {\it natural self-similar measure} on $K$. See  \cite{hutchinson} for the  proofs of the above facts.

Nevertheless, given a self-similar set $K\subset \R^d$ satisfying the OSC with $\dim_{\rm H}K=s$, it is often challenging to determine the exact value of $\mathcal{H}^s(K)$. When $d=1$, this problem was first studied independently by Marion \cite{marion86,marion87} and  Ayer and Strichartz \cite{ayerstr}. They computed the exact value of $\mathcal{H}^s(K)$ under certain additional hypothesis. Their method is based on the relation \eqref{eq:mu and Hs} and the convex density theorem for Hausdorff measures (see \cite[Theorem 2.3]{Falconer85}).    When $d>1$ and $s>1$ is not an integer, despite  numerous studies,  not a single  example is known  for which the exact value of $\mathcal{H}^s(K)$ is computed; see \cite{DHL,DHLT,DRZ, ZhouWu} for related works.  In this case,  it is known that  $\mathcal{H}^s(K)<|K|^s$  if the convex hull of $K$ is a polytope (see \cite[Corollary 1.6]{HLZ05}). It remains an open question whether $\mathcal{H}^s(K)<|K|^s$ for any self-similar set $K\subset \R^d$ with similarity dimension $s>1$.  Regarding of this problem,  Zhiying Wen raised the following question in \cite{Wen}.

\begin{ques}
Let $s>1$ and $\epsilon\in(0,1)$. Does there exist a self-similar set $K$ with similarity dimension $s$  such that $\mathcal{H}^s(K)>(1-\epsilon)|K|^s$?
\footnote{It is known that for each $s\in (0,1]$, there exists a self-similar set $K\subset {\Bbb R}$ with similarity dimension $s$ such that $\mathcal{H}^s(K)=|K|^s$ (see e.g. \cite{marion86, ayerstr}).}
\end{ques}

In this paper, we give an affirmative  answer to the above  question  by proving the following.

\begin{thm}\label{mainthm}
For every $s\in (0,d)$ and $c\in (0,1)$, there exists a self-similar set $K\subset\R^d$ so that its generating IFS  satisfies the SSC,  $\dim_{\rm H}K=s$  and $\mathcal{H}^s(K)=c|K|^s$.
\end{thm}

 Our strategy of the proof is as follows: given $s\in (0,d)$ and $\epsilon\in(0,1)$, we construct a  family of IFSs $\{\Phi_t\}_{t\in [0,1]}$ with the corresponding self-similar sets $\{K_t\}_{t\in [0,1]}$ such that:  (1) for each $t\in [0,1]$, $\Phi_t$ satisfies the SSC, $\dim_{\rm H}K_t=s$ and $|K_t|=1$; (2) the mapping $t\mapsto \mathcal{H}^s(K_t)$ is continuous; (3) $\mathcal{H}^s(K_0)>1-\epsilon$;  (4) $\mathcal{H}^s(K_1)<\epsilon$. See Proposition \ref{mainprop} for the details. A key part is the proof of (3), in which  we apply the isodiametric inequality (see Lemma \ref{lemiso}). Then Theorem \ref{mainthm} follows from the above  result and  a result in \cite{olsen} about the continuity of Hausdorff measures of self-similar sets satisfying the  SSC with respect to the defining data of the IFSs.

\setcounter{equation}{0}

\section{Proof of  Theorem \ref{mainthm}}

Our proof of Theorem \ref{mainthm} is based  on the following result.
\begin{lem}\label{lem:density}\cite{olsen2}
Let $K\subset \R^d$ be a self-similar set generated by an IFS $\Phi=\{\phi_i\}_{i=1}^{\ell}$ which satisfies the SSC with $\dim_{\rm H}K=s$. Let $\mu$ be the natural self-similar measure on $K$. Set $\Delta=\min_{i\neq j}{\rm dist}(\phi_i(K), \phi_j(K))$. Then we have
\begin{align}
\mathcal{H}^s(K)^{-1}&=\max\left\{\frac{\mu(U)}{|U|^s}:  U\subset \R^d\text{ is compact and  convex}\right\} \label{eqdensity1} \\
&=\max\left\{\frac{\mu(U)}{|U|^s}:  U\subset \R^d\text{ is compact and  convex with } |U|\geq \Delta\right\}. \label{eqdensity2}
\end{align}
In particular, the above maximums are both attained.
\end{lem}

We remark that when $d=1$, Lemma \ref{lem:density} was proved earlier in \cite{marion86,stz,ayerstr}, where it was   used  to compute  Hausdorff measures of self-similar sets in $\R$ under certain additional hypothesis. Moreover, Lemma \ref{lem:density} is not explicitly stated in \cite{olsen2}, but it can be easily deduced from the results of \cite{olsen2}. Indeed, it was proved in  \cite[Corollaries 1.5-1.6]{olsen2}  that
\begin{align*}
\mathcal{H}^s(K)^{-1}&=\sup\left\{\frac{\mu(U)}{|U|^s}:  U\subset \R^d\text{ is open  and  convex}\right\}\\
&=\sup\left\{\frac{\mu(U)}{|U|^s}:  U\subset \R^d\text{ is open and  convex with } |U|\geq \Delta\right\}.
\end{align*}
Then Lemma \ref{lem:density}  follows from the above equalities together with  a standard compactness argument.

Based on Lemma \ref{lem:density}, we establish the following result, which is  an essential part in  our proof of Theorem \ref{mainthm}.

\begin{prop}\label{mainprop}
For every $s\in (0,d)$ and $\epsilon\in (0,1)$,  there exist $r\in (0,1)$, $\ell\in \N$ and  a family of IFSs $\Phi_t=\{\phi_{t,i}(x)=rx+a_i(t)\}_{i=1}^{\ell}$ ($t\in [0,1]$) on $\R^d$ with  $a_1,\ldots, a_{\ell}: [0,1]\to \R^d$ being continuous functions such that the following statements hold:

{\rm (i)} For each $t\in [0,1]$, $\Phi_t$ satisfies the SSC and  its attractor, denoted as $K_t$, has dimension $s$ and diameter $1$.

{\rm (ii)} $\mathcal{H}^s(K_0)>1-\epsilon$ and $\mathcal{H}^s(K_1)<\epsilon$.
\end{prop}

To prove Proposition \ref{mainprop}, we first give an elementary lemma.

\begin{lem}\label{lem2}
Let $0<s<d$ and $\epsilon>0$.  Then there exists $N>0$ such that  for all  $n\geq N$ and all   $x\in [\frac{1}{4n},1]$,
\[\frac{\left(x+\frac{\sqrt{d}}{n}\right)^d}{\left(1-\frac{\sqrt{d}}{2n}\right)^dx^s}<1+\epsilon.\]
\end{lem}

\begin{proof}
Let $y_0>0$ be large enough such that for all $y\geq y_0$,
\begin{equation}\label{eq1e1}
\frac{\left(1+\frac{\sqrt{d}}{y}\right)^d}{\left(1-\frac{\sqrt{d}}{y}\right)^d}<1+\epsilon.
\end{equation}
Then take $N>0$  large enough such that  for all $n\geq N$,
\begin{equation}\label{eq1e2}
\frac{\left(1+4\sqrt{d}\right)^d}{\left(1-\frac{\sqrt{d}}{2n}\right)^d}\left(\frac{y_0}{n}\right)^{d-s}<1+\epsilon.
\end{equation}
Notice that such $N$ exists since $s<d$. Let $n\geq N$ and $x\in [\frac{1}{4n},1]$. If $nx\geq y_0$, then $2n>nx\geq y_0$ and so by \eqref{eq1e1},
\[\frac{\left(x+\frac{\sqrt{d}}{n}\right)^d}{\left(1-\frac{\sqrt{d}}{2n}\right)^dx^s}=\frac{\left(1+\frac{\sqrt{d}}{nx}\right)^d}{\left(1-\frac{\sqrt{d}}{2n}\right)^d}x^{d-s}\leq \frac{\left(1+\frac{\sqrt{d}}{nx}\right)^d}{\left(1-\frac{\sqrt{d}}{nx}\right)^d}x^{d-s}<1+\epsilon.\]
If $nx<y_0$, then
\begin{align*}
\frac{\left(x+\frac{\sqrt{d}}{n}\right)^d}{\left(1-\frac{\sqrt{d}}{2n}\right)^dx^s}&\leq \frac{\left(x+4\sqrt{d}x\right)^d}{\left(1-\frac{\sqrt{d}}{2n}\right)^dx^s} \quad \quad \quad\quad (\text{since } \frac{1}{n}\leq 4x)\\
&= \frac{\left(1+4\sqrt{d}\right)^d}{\left(1-\frac{\sqrt{d}}{2n}\right)^d}x^{d-s}\\
&\leq \frac{\left(1+4\sqrt{d}\right)^d}{\left(1-\frac{\sqrt{d}}{2n}\right)^d}\left(\frac{y_0}{n}\right)^{d-s} \quad (\text{since }nx<y_0)\\
&<1+\epsilon   \ \ \qquad\qquad\qquad\quad\qquad(\text{by }\eqref{eq1e2}).
\end{align*}
This completes the proof of the lemma.
\end{proof}

Let $\mathcal{L}^d$ denote the $d$-dimensional Lebesgue measure on $\R^d$. The following standard isodiametric inequality plays a key role in our proof of Proposition \ref{mainprop}.
\begin{lem}\cite[Theorem 2.4]{eg}\label{lemiso}
For every Lebesgue measurable set $A\subset \R^d$,
\[\mathcal{L}^d(A)\leq \omega_d2^{-d}|A|^d,\]
where $\omega_d$ denotes the Lebesgue measure of a unit ball in $\R^d$.
\end{lem}

For $x=(x_1,\ldots,x_d)\in \R^d$ and $\delta>0$, let  $B(x,\delta)$ be the closed ball in $\R^d$ centered at $x$ of radius $\delta$, and let $Q(x,\delta)$ denote the cube $\prod_{i=1}^d[x_i-\delta,x_i+\delta]$.  For $A\subset \R^d$,  let $\#A$ be  the cardinality of $A$.

\begin{proof}[Proof of Proposition \ref{mainprop}]
Fix $s\in (0,d)$ and $\epsilon\in (0,1)$.  We are going to construct the  self-similar sets $K_t$ $(t\in [0,1])$ in the ball $B:=B(0,1/2)$. For this purpose, let $N$ be as in Lemma \ref{lem2}. Pick a positive integer $n\geq N$ so that
\begin{equation}\label{eq:n}
\frac{\omega_d\left(n-\frac{\sqrt{d}}{2}\right)^d-2}{(8n+4)^s}>\epsilon^{-1}.\end{equation}
  Set $F=B\cap (\Z^d/(2n)).$ Let $\ell=\#F$ and $b_1,\ldots, b_{\ell}$ be the different elements of $F$ with  $$b_1=(-1/2,0,\ldots, 0),\quad  b_2=(1/2,0,\ldots, 0).$$
   Then by  a simple volume argument (see e.g. \cite[p. 17]{Kratzel}),
\begin{equation}\label{eq:bdd ln}
\ell\geq \omega_d\left(n-\frac{\sqrt{d}}{2}\right)^d.
\end{equation}
Let $r=\ell^{-1/s}$. Then by \eqref{eq:n}-\eqref{eq:bdd ln},
 \begin{equation}\label{eq:rn}
 (8n+4)r=(8n+4)\ell^{-1/s}\leq (8n+4)\omega_d^{-1/s}\left(n-\frac{\sqrt{d}}{2}\right)^{-d/s}<\epsilon^{1/s}<1.
 \end{equation}
 In particular, $8nr<1$.

For $t\in [0,1]$,  let $\Phi_t=\{\phi_{t,i}(x)=rx+a_i(t)\}_{i=1}^{\ell}$, where
\[a_1(t)=(1-r)b_1,\; a_2(t)=(1-r)b_2,\; a_{i}(t)=(1-r)(8nr)^tb_i \quad \text{ for }  i\in \{3,\ldots,\ell\}.\] Let $K_t$ be the  attractor of $\Phi_t$. Notice that  $a_1,\ldots, a_{\ell}:[0,1]\to\R^d$ are continuous functions. Clearly,  for each $1\leq i\leq \ell$, the fixed point of $\phi_{t,i}$ is $a_i(t)/(1-r)$. In particular, $b_1, b_2$ are the fixed points of  $\phi_{t,1}$ and $\phi_{t,2}$,  respectively. Hence $b_1,b_2\in K_t$.     Since $8nr<1$,  it is not difficult to check   that for each $t\in [0,1]$, $\phi_{t,i}(B)$ ($i=1,\ldots, \ell$) are pairwise disjoint and contained in $B$. This implies that $K_t\subset B$ (see e.g. \cite{falconer}) and so $|K_t|\leq |B|=1$, and $\Phi_t$ satisfies the SSC with $\dim_{\rm H}K_t=s$. Since   $b_1, b_2\in K_t$,  $|K_t|\geq\|b_1-b_2\|=1.$ Hence  $|K_t|=1$. This proves the part (i) of the proposition. In the following we  prove that  $\mathcal{H}^s(K_0)>1-\epsilon$ and $\mathcal{H}^s(K_1)<\epsilon$.

 For $t\in\{0,1\}$,  let $\mu_t$ be the natural self-similar measure on $K_t$. That is, $\mu_t$ is the unique Borel probability measure supported on $K_t$ satisfying
\begin{equation}\label{eq:similar id}
\mu_t=\sum_{i=1}^{\ell}r^s\mu_{t}\circ(\phi_{t,i})^{-1}.
\end{equation}
We first show that  $\mathcal{H}^s(K_0)>1-\epsilon$. Recall that $\Phi_0=\{\phi_{0, i}(x)=rx+(1-r)b_i\}_{i=1}^{\ell}$.
See Figure \ref{fig:1}(a) for an illustration of the locations of $b_1,\ldots, b_{\ell}$, which are the fixed points of the elements of $\Phi_0$. \begin{figure}[htbp]
\centering
\setlength{\abovecaptionskip}{0.cm}
\subfigure[The bold dots are the  elements of $F$  and also the fixed points of the similitudes  in $\Phi_0$]{\label{subfiga}
\begin{minipage}{7cm}
\centering
\begin{tikzpicture}[scale=5.2]
\draw (-1/3, -1/2)--(-1/3,1/2);
\draw (-1/6, -1/2)--(-1/6,1/2);
\draw[->] (0, -1/2-1/9)--(0,1/2+1/9);
\draw (1/6, -1/2)--(1/6,1/2);
\draw (1/3, -1/2)--(1/3,1/2);
\draw (-1/2, -1/3)--(1/2,-1/3);
\draw (-1/2, -1/6)--(1/2,-1/6);
\draw[->] (-1/2-1/9, 0)--(1/2+1/9,0);
\draw (-1/2, 1/6)--(1/2,1/6);
\draw (-1/2, 1/3)--(1/2,1/3);

\draw(0,0) circle (1/2);
\node[below,right] at (-0.02,-0.03) {$0$};
\node[below,right] at (1/2-0.01,-0.05) {$\frac{1}{2}$};
\node[above,right] at (1/2-0.01,0.04) {$b_1$};
\node[above,left] at (-1/2+0.01,0.04) {$b_2$};
\node[below,right] at (1/2+1/12,-0.03) {$x$};
\node[below,right] at (0,1/2+1/12) {$y$};

 \fill (0,0) circle [radius=0.25pt];
 \fill (-1/2,0) circle [radius=0.25pt];
 \fill (1/2,0) circle [radius=0.25pt];
 \fill (0,-1/2) circle [radius=0.25pt];
 \fill (0,1/2) circle [radius=0.25pt];
 \fill (-1/3,-1/3) circle [radius=0.25pt];
 \fill (-1/3,-1/6) circle [radius=0.25pt];
 \fill (-1/3,0) circle [radius=0.25pt];
 \fill (-1/3,1/6) circle [radius=0.25pt];
 \fill (-1/3,1/3) circle [radius=0.25pt];
\fill (-1/6,-2/6) circle [radius=0.25pt];
 \fill (-1/6,-1/6) circle [radius=0.25pt];
 \fill (-1/6,0) circle [radius=0.25pt];
 \fill (-1/6,1/6) circle [radius=0.25pt];
 \fill (-1/6,2/6) circle [radius=0.25pt];
\fill (0,-2/6) circle [radius=0.25pt];
 \fill (0,-1/6) circle [radius=0.25pt];
 \fill (0,1/6) circle [radius=0.25pt];
 \fill (0,2/6) circle [radius=0.25pt];
\fill (1/6,-2/6) circle [radius=0.25pt];
 \fill (1/6,-1/6) circle [radius=0.25pt];
 \fill (1/6,0) circle [radius=0.25pt];
 \fill (1/6,1/6) circle [radius=0.25pt];
 \fill (1/6,2/6) circle [radius=0.25pt];
\fill (2/6,-2/6) circle [radius=0.25pt];
 \fill (2/6,-1/6) circle [radius=0.25pt];
 \fill (2/6,0) circle [radius=0.25pt];
 \fill (2/6,1/6) circle [radius=0.25pt];
 \fill (2/6,2/6) circle [radius=0.25pt];
\end{tikzpicture}
\end{minipage}}\hspace{.2in}
\subfigure[The bold dots are the fixed points of the similitudes in $\Phi_1$]{\label{subfigb}
\begin{minipage}{7cm}
\centering
\begin{tikzpicture}[scale=5.2]
\draw[->] (-1/2-1/9, 0)--(1/2+1/9,0);
\draw[->] (0, -1/2-1/9)--(0,1/2+1/9);
\node[below,right] at (1/2+1/12,-0.03) {$x$};
\node[below,right] at (0,1/2+1/12) {$y$};
\draw(0,0) circle (1/2);

\draw(0,0) circle(1.6/8);
\draw(-1/8+1/24, -1/8)--(-1/8+1/24, 1/8);
\draw(-1/8+2/24, -1/8)--(-1/8+2/24, 1/8);
\draw(-1/8+3/24, -1/8)--(-1/8+3/24, 1/8);
\draw(-1/8+4/24, -1/8)--(-1/8+4/24, 1/8);
\draw(-1/8+5/24, -1/8)--(-1/8+5/24, 1/8);
\draw(-1/8, -1/8+1/24)--(1/8, -1/8+1/24);
\draw(-1/8, -1/8+2/24)--(1/8, -1/8+2/24);
\draw(-1/8, -1/8+3/24)--(1/8, -1/8+3/24);
\draw(-1/8, -1/8+4/24)--(1/8, -1/8+4/24);
\draw(-1/8, -1/8+5/24)--(1/8, -1/8+5/24);
\fill (0,0) circle [radius=0.25pt];
 \fill (-1/2,0) circle [radius=0.25pt];
 \fill (1/2,0) circle [radius=0.25pt];
 \fill (-1/8+1/24,-1/8+1/24) circle [radius=0.25pt];
 \fill (-1/8+1/24,-1/8+2/24) circle [radius=0.25pt];
 \fill (-1/8+1/24,-1/8+3/24) circle [radius=0.25pt];
 \fill (-1/8+1/24,-1/8+4/24) circle [radius=0.25pt];
 \fill (-1/8+1/24,-1/8+5/24) circle [radius=0.25pt];
\fill (-1/8+2/24,-1/8+1/24) circle [radius=0.25pt];
 \fill (-1/8+2/24,-1/8+2/24) circle [radius=0.25pt];
 \fill (-1/8+2/24,-1/8+3/24) circle [radius=0.25pt];
 \fill (-1/8+2/24,-1/8+4/24) circle [radius=0.25pt];
 \fill (-1/8+2/24,-1/8+5/24) circle [radius=0.25pt];

 \fill (-1/8+3/24,-1/8) circle [radius=0.25pt];
\fill (-1/8+3/24,-1/8+1/24) circle [radius=0.25pt];
 \fill (-1/8+3/24,-1/8+2/24) circle [radius=0.25pt];
 \fill (-1/8+3/24,-1/8+3/24) circle [radius=0.25pt];
 \fill (-1/8+3/24,-1/8+4/24) circle [radius=0.25pt];
 \fill (-1/8+3/24,-1/8+5/24) circle [radius=0.25pt];
  \fill (-1/8+3/24,-1/8+6/24) circle [radius=0.25pt];

\fill (-1/8+4/24,-1/8+1/24) circle [radius=0.25pt];
 \fill (-1/8+4/24,-1/8+2/24) circle [radius=0.25pt];
 \fill (-1/8+4/24,-1/8+3/24) circle [radius=0.25pt];
 \fill (-1/8+4/24,-1/8+4/24) circle [radius=0.25pt];
 \fill (-1/8+4/24,-1/8+5/24) circle [radius=0.25pt];

\fill (-1/8+5/24,-1/8+1/24) circle [radius=0.25pt];
 \fill (-1/8+5/24,-1/8+2/24) circle [radius=0.25pt];
 \fill (-1/8+5/24,-1/8+3/24) circle [radius=0.25pt];
 \fill (-1/8+5/24,-1/8+4/24) circle [radius=0.25pt];
 \fill (-1/8+5/24,-1/8+5/24) circle [radius=0.25pt];

\node[left] at (0.327-1/2,0.1) {$V$};
\node[below,right] at (1/2-0.01,-0.05) {$\frac{1}{2}$};
\node[above,right] at (1/2-0.01,0.04) {$b_1$};
\node[above,left] at (-1/2+0.01,0.04) {$b_2$};

\node at (1/50,1/50) {{\tiny $0$}};
\end{tikzpicture}
\end{minipage}}
\caption{$\Phi_0$ and $\Phi_1$, $n=3$}
\label{fig:1}
\end{figure}
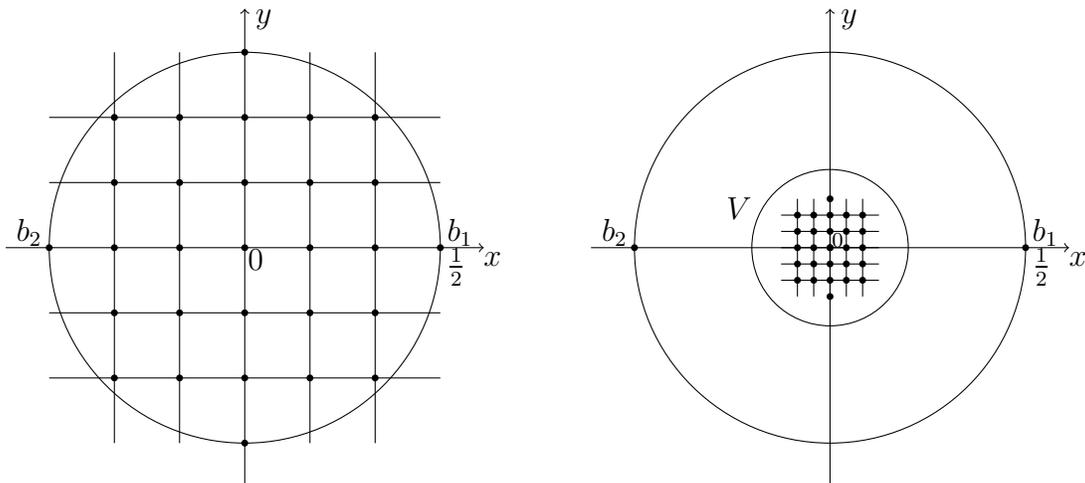
 Since $|K_0|=1$ and $8nr<1$, we see that for $i\in\{1,\ldots,\ell\}$,  $\phi_{0,i}(K_0)$ is contained in the interior of $Q(b_i, 1/(4n))$. In particular, $Q(b_i, 1/(4n))\cap \phi_{0,j}(K_0)=\emptyset$ for any $i,j\in\{1,\ldots,\ell\}$ with $i\neq j$.    Hence by  \eqref{eq:similar id},
\begin{equation}\label{eq:mu on Q}
\mu_0(Q(b_i, 1/(4n)))=r^s=\ell^{-1}, \ \ \  i=1,\ldots, \ell.
\end{equation}

By \eqref{eqdensity2}  there exists a compact convex set $U\subset \R^d$ such that
\begin{equation}\label{eq:hs(k) dmax}
\frac{\mu_0(U)}{|U|^s}=\mathcal{H}^s(K_0)^{-1}
\end{equation}
with $|U|\geq\min_{i\neq j}{\rm dist}(\phi_{0,i}(K_0), \phi_{0,j}(K_0))$. Notice that for $i\in\{1,\ldots,\ell\}$, $b_i$ is the fixed point of $\phi_{0,i}$ and so $b_i\in \phi_{0,i}(K_0)$. Hence for all $i\neq j$, since $|\phi_{0,i}(K_0)|=|\phi_{0,j}(K_0)|=r$ and $8nr<1$, we have by the triangle inequality,
\begin{align*}
{\rm dist}(\phi_{0,i}(K_0), \phi_{0,j}(K_0))&\geq \|b_i-b_j\|-2r\geq \frac{1}{2n}-2r>\frac{1}{4n}.
\end{align*}
It follows that $|U|\geq \frac{1}{4n}$.
On the other hand, since $K_0\subset B$, we have $\mu_0(U)=\mu_0(U\cap B)$ and so
\begin{equation}\label{eqUintB}
\frac{\mu_0(U)}{|U|^s}\leq \frac{\mu_0(U\cap B)}{|U\cap B|^s}.
\end{equation}
However, since $U\cap B$ is compact and  convex, by \eqref{eqdensity1}  and \eqref{eq:hs(k) dmax} we see that the equality holds in \eqref{eqUintB}. Therefore $|U|=|U\cap B|$.
Hence  replacing $U$ by $U\cap B$ if necessary, we can assume that $U\subset B$ and thus $|U|\leq 1$. So we have
\begin{equation}\label{eqbdU}
\frac{1}{4n}\leq |U|\leq 1.
\end{equation}

Let $m=\#\mathcal{F}$, where $\mathcal{F}:=\{Q(b,1/(4n)):b\in F, \  Q(b,1/(4n))\cap U\neq\emptyset\}$. Then by \eqref{eq:mu on Q}  and  \eqref{eq:bdd ln},
\begin{equation}\label{eq:mu(U)}
\mu_0(U)\leq m\ell^{-1}\leq m\omega_d^{-1}\left(n-\frac{\sqrt{d}}{2}\right)^{-d}.
\end{equation}
On the other hand, notice that   each cube in $\mathcal{F}$ is of diameter $\sqrt{d}/(2n)$  and so it is contained in the closed $\sqrt{d}/(2n)$-neighborhood of $U$, which we denote by $\overline{\mathbf{V}}_{\sqrt{d}/(2n)}(U)$, and the intersection of any two different cubes in $\mathcal{F}$  has zero Lebesgue measure. Hence by a simple volume argument and the isodiametric inequality (see Lemma \ref{lemiso}),
\begin{equation}\label{eq:leb}
m\left(\frac{1}{2n}\right)^{d}\leq \mathcal{L}^d\left(\overline{\mathbf{V}}_{\sqrt{d}/(2n)}(U)\right)\leq \omega_d2^{-d}\left(|U|+\frac{\sqrt{d}}{n}\right)^d.
\end{equation}
 Now by  \eqref{eq:mu(U)}-\eqref{eq:leb}  and Lemma \ref{lem2} (in which we take $|U|=x$ and recall \eqref{eqbdU}),
\begin{equation*}\label{eq:mu(U)2}
\frac{\mu_0(U)}{|U|^s}\leq\frac{\left(|U|+\frac{\sqrt{d}}{n}\right)^d}{\left(1-\frac{\sqrt{d}}{2n}\right)^d|U|^s}<1+\epsilon.
\end{equation*}
This combining with \eqref{eq:hs(k) dmax} yields that $\mathcal{H}^s(K_0)>1/(1+\epsilon)>1-\epsilon$.

Finally,  we show  that $\mathcal{H}^s(K_1)<\epsilon$. Recall that $\Phi_1$ consists of the similitudes $\phi_{1,1}(x)=rx+(1-r)b_1$, $\phi_{1,2}(x)=rx+(1-r)b_2$ and $\phi_{1,k}(x)=rx+(1-r)8nrb_k$ for $k\in\{3, \ldots, \ell\}$. Let $V=B(0, (4n+1)r)$. Since $(8n+4)r<1$ (see \eqref{eq:rn}),  it is easily checked that   $\phi_{1,1}(K_1)$, $\phi_{1,2}(K_1)$ are both disjoint from $V$, and $\phi_{1,k}(K_1)\subset V$ for $k\in \{3,\ldots,\ell\}$.  See Figure \ref{fig:1}(b) for an illustration of $V$ and  the locations of the fixed points of the elements of  $\Phi_1$. Then by \eqref{eq:n}-\eqref{eq:bdd ln},
\begin{equation*}\label{eq:d U'}
\frac{\mu_1(V)}{|V|^s}=\frac{1-2r^s}{(8n+2)^sr^s}=\frac{\ell-2}{(8n+2)^s}\geq \frac{\omega_d\left(n-\frac{\sqrt{d}}{2}\right)^d-2}{(8n+2)^s}>\epsilon^{-1}.
\end{equation*}
Hence $\mathcal{H}^s(K_1)<\epsilon$ by  Lemma \ref{lem:density}.  This completes the proof of the  proposition.
\end{proof}

With Proposition \ref{mainprop} in hand, we are ready to prove Theorem \ref{mainthm}. The proof is a direct consequence of Proposition \ref{mainprop} combined with a continuity result in \cite{olsen} about  Hausdorff measures of self-similar sets generated by IFSs satisfying the SSC.

\begin{proof}[Proof of Theorem \ref{mainthm}]
Fix $s\in (0,d)$ and $c\in (0,1)$. Let $\epsilon>0$ be so small  that $\epsilon<c<1-\epsilon$. Let $\Phi_t, K_t$  ($t\in [0,1]$) be constructed as in Proposition \ref{mainprop}. Since $a_1,\ldots,a_{\ell}:[0,1]\to\R^d$ are continuous functions,  we easily see  from   \cite[Theorem 1.2]{olsen} that the mapping  $t\mapsto \mathcal{H}^s(K_t)$ is continuous on $[0,1]$. Since $\mathcal{H}^s(K_0)>1-\epsilon$, $\mathcal{H}^s(K_1)<\epsilon$ and $c\in(\epsilon, 1-\epsilon)$, the continuity of $t\mapsto \mathcal{H}^s(K_t)$ implies that $\mathcal{H}^s(K_{t_0})=c$ for some $t_0\in [0,1]$. Therefore, $\mathcal{H}^s(K_{t_0})=c|K_{t_0}|^s$ as $|K_{t_0}|=1$. Letting $K=K_{t_0}$ we complete the proof of Theorem \ref{mainthm}.
\end{proof}

{\noindent \bf  Acknowledgements}. The authors are grateful to their supervisor, De-Jun Feng, for many helpful discussions and suggestions. They also thank the two anonymous referees for their suggestions and comments to  improve the paper.  This research was partially supported by  a
HKRGC GRF grant (project 14301017) and the Direct Grant for Research in CUHK.

\begin{multicols}{2}
{\it \noindent Cai-Yun Ma \\
Department of Mathematics\\
The Chinese University of Hong Kong\\
Shatin\\
Hong Kong}\\
\href{mailto:cyma@math.cuhk.edu.hk}{cyma@math.cuhk.edu.hk}
\columnbreak
 \\
{\it Yu-Feng Wu \\
Department of Mathematics\\
South China University of Technology\\
Guangzhou 510641\\
P.R. China}\\ 
and \\
{\it Department of Mathematics\\
The Chinese University of Hong Kong\\
Shatin\\
Hong Kong}\\
\href{mailto:yufengwu@scut.edu.cn}{yufengwu@scut.edu.cn}\\
\href{mailto: yufengwu.wu@gmail.com}{yufengwu.wu@gmail.com}

\end{multicols}

\end{document}